\documentclass{article}

\usepackage{graphicx}
\usepackage{setspace}
\usepackage{enumerate,lineno,setspace,float}
\usepackage[small,compact]{titlesec}
\usepackage{amsmath,amsfonts,amssymb,amsthm}
\usepackage{geometry}
\usepackage{fancyhdr}
\usepackage{amssymb,amsmath}

\newtheorem{Theorem}{Theorem}[section]
\newtheorem{Definition}[Theorem]{Definition}

\newtheorem{Corollary}[Theorem]{Corollary}
\newtheorem{Lemma}[Theorem]{Lemma}

\newtheorem{Remark}[Theorem]{Remark}

\usepackage{latexsym}
\usepackage{graphicx, float}
\usepackage{amssymb,amsfonts,amsthm,amsmath,graphicx,url}

\usepackage{color}

\definecolor{VeryLightBlue}{rgb}{0.9,0.9,1}
\definecolor{LightBlue}{rgb}{0.8,0.8,1}
\definecolor{MidBlue}{rgb}{0.5,0.5,1}
\definecolor{DarkBlue}{rgb}{0,0,0.6}
\definecolor{Blue}{rgb}{0,0,1}

\definecolor{Gold}{rgb}{1,0.843,0}

\definecolor{LightGreen}{rgb}{0.88,1,0.88}
\definecolor{MidGreen}{rgb}{0.6,1,0.6}
\definecolor{DarkGreen}{rgb}{0,0.6,0}

\definecolor{VeryLightYellow}{rgb}{1,1,0.9}
\definecolor{LightYellow}{rgb}{1,1,0.6}
\definecolor{MidYellow}{rgb}{1,1,0.5}
\definecolor{DarkYellow}{rgb}{1,1,0.2}

\definecolor{DarkPurple}{rgb}{.6,0,1}

\definecolor{Red}{rgb}{1,0,0}
\definecolor{VeryLightRed}{rgb}{1,0.9,0.9}
\definecolor{LightRed}{rgb}{1,0.8,0.8}
\definecolor{MidRed}{rgb}{1,0.55,0.55}

\long\def\delete#1{}

\input amssym.def
\input amssym.tex

\newcommand{\be}{\begin{equation}}
\newcommand{\ee}{\end{equation}}
\newcommand{\bea}{\begin{eqnarray}}
\newcommand{\eea}{\end{eqnarray}}
\newcommand{\bean}{\begin{eqnarray*}}
\newcommand{\eean}{\end{eqnarray*}}
\def\non{\nonumber}

\def\qed{\hfill$\Box$\vspace{11pt}}

\def\diam{{\rm diam}}
\def\span{{\rm span}}
\def\rn{{\rm rn}}

\def\d{\delta}

\def\l{\lambda}

\def\ve{\varepsilon}

\def\({\left(}
\def\){\right)}
\def\[{\left[}
\def\]{\right]}

\begin{document}

\title{Radio number of trees}
\author{Devsi Bantva \\ Department of Mathematics \\ Lukhdhirji Engineering College, Morvi 363 642, Gujarat, India \\ \textit{devsi.bantva@gmail.com} \\ \\
Samir Vaidya \\ Department of Mathematics \\ Saurashtra University, Rajkot 360 005, Gujarat, India \\  \textit{samirkvaidya@yahoo.co.in} \\ \\
Sanming Zhou \\
School of Mathematics and Statistics\\
The University of Melbourne,
Parkville, VIC 3010, Australia\\
\textit{smzhou@ms.unimelb.edu.au}}


\date{}
\openup 0.48\jot
\maketitle

\begin{abstract}
A radio labeling of a graph $G$ is a mapping $f: V(G) \rightarrow \{0, 1, 2, \ldots\}$ such that $|f(u)-f(v)|\geq \diam(G) + 1 - d(u,v)$ for every pair of distinct vertices $u, v$ of $G$, where $\diam(G)$ is the diameter of $G$ and $d(u,v)$ the distance between $u$ and $v$ in $G$. The radio number of $G$ is the smallest integer $k$ such that $G$ has a radio labeling $f$ with $\max\{f(v) : v \in V(G)\} = k$. We give a necessary and sufficient condition for a lower bound on the radio number of trees to be achieved, two other sufficient conditions for the same bound to be achieved by a tree, and an upper bound on the radio number of trees. Using these, we determine the radio number for three families of trees.


\end{abstract}

\section{Introduction}
\label{sec:int}

In a graph model for the channel assignment problem, the transmitters are represented by the vertices of a graph; two vertices are adjacent or at distance two apart in the graph if the corresponding transmitters are \emph{very close} or \emph{close} to each other. Motivated by this problem Griggs and Yeh \cite{Griggs} introduced the following distance-two labeling problem: An \emph{$L(2,1)$-labeling} of a graph $G=(V(G),E(G))$ is a function $f$ from the vertex set $V(G)$ to the set of nonnegative integers such that $|f(u)-f(v)|\geq2$ if $d(u,v)=1$ and $|f(u)-f(v)|\geq1$ if $d(u,v)=2$, where $d(u, v)$ is the distance between $u$ and $v$ in $G$. The \emph{span} of $f$ is defined as $\max\{f(u)-f(v): u, v \in V(G)\}$, and the minimum span over all $L(2,1)$-labelings of $G$ is called the \emph{$\lambda$-number} of $G$, denoted by $\l(G)$. The $L(2,1)$-labeling and other distance-two labeling problems have been studied by many researchers in the past two decades; see \cite{Calamoneri} and \cite{Yeh1}.

It has been observed that interference among transmitters may go beyond two levels. Motivated by the channel assignment problem for FM radio stations, Chartrand \emph{et al.} \cite{Chartrand1} introduced the following radio labeling problem. Denote by $\diam(G)$ the \emph{diameter} of $G$, that is, the maximum distance among all pairs of vertices in $G$.

\begin{Definition}
{\em
A \emph{radio labeling} of a graph $G$ is a mapping $f: V(G) \rightarrow \{0, 1, 2, \ldots\}$ such that for every pair of distinct vertices $u, v$ of $G$,
$$
d(u,v) + |f(u)-f(v)| \geq \diam(G) + 1.
$$
The integer $f(u)$ is called the \emph{label} of $u$ under $f$, and the \emph{span} of $f$ is defined as $\span(f) = \max \{|f(u)-f(v)|: u, v \in V(G)\}$. The \emph{radio number} of $G$ is defined as
$$
\rn(G) := \min_{f} \span(f)
$$
with minimum over all radio labelings $f$ of $G$. A radio labeling $f$ of $G$ is \emph{optimal} if $\span(f) = \rn(G)$.
}
\end{Definition}

Observe that any radio labeling should assign different labels to distinct vertices. Note also that any optimal radio labeling must assign $0$ to some vertex. In the case when $\diam(G) = 2$ we have $\rn(G) = \l(G)$.

Determining the radio number of a graph is an interesting but challenging problem. So far the radio number is known only for a handful families of graphs (see \cite{Chartrand} for a survey). Chartrand \emph{et al.} \cite{Chartrand1,Chartrand2,Zhang} studied the radio labeling problem for paths and cycles, and this was continued by Liu and Zhu \cite{Liu} who gave the exact value of the radio number for paths and cycles. We emphasise that even in these innocent-looking cases it was challenging to determine the radio number. In \cite{Daphne2,Daphne3}, Liu and Xie discussed the radio number for the square of paths and cycles. In \cite{Vaidya1,Vaidya2,Vaidya3}, Vaidya and Bantva studied the radio number for the total graph of paths, the strong product of $P_{2}$ with $P_{n}$ and linear cacti. In \cite{Benson}, Benson \emph{et al.} determined the radio number of all graphs of order $n$ and diameter $n-2$, where $n \ge 2$ is an integer. Bhatti \emph{et al.} studied \cite{Bhatti} the radio number of wheel-like graphs, while \v{C}ada \emph{et al.} discussed \cite{Cada} a general version of radio labelings of distance graphs. In \cite{Daphne1}, Liu gave a lower bound on the radio number for trees and presented a class of trees achieving this bound. In \cite{Li}, Li \emph{et al.} determined the radio number for complete $m$-ary trees. In \cite{Tuza}, Hal\'asz and Tuza determined the radio number of internally regular complete trees among other things. (A few distance-three labeling problems for such trees of even diameters were studied in \cite{KLZ}.) In spite of these efforts, the problem of determining the exact value of the radio number for trees is still open, and it seems unlikely that a universal formula exists for all trees.

Inspired by the work in \cite{Li}, in this paper we first give a necessary and sufficient condition for a lower bound \cite[Theorem 3]{Daphne1} (see also Lemma \ref{thm:lb}) on the radio number of trees to be achieved (Theorem \ref{thm:ub}), together with an optimal radio labeling. We also give two sufficient conditions for this bound to be achieved (Theorem \ref{thm:cor1}) and obtain an upper bound on the radio number of trees (Theorem \ref{thm:ub1}). These results provide methodologies for obtaining the exact values of or upper bounds on the radio number of trees, and using them we determine in Section \ref{sec:three} the radio number for three families of trees, namely banana trees, firecracker trees, and caterpillars in which all vertices on the spine have the same degree. Our result for caterpillars implies the result in \cite{Liu} for paths. As concluding remarks, in Section \ref{sec:rem} we demonstrate that the results on the radio numbers of internally regular complete trees (\cite[Theorem 1]{Tuza}) and complete $m$-ary trees for $m \ge 3$ (\cite[Theorem 2]{Li}) can be obtained by using our method.

\section{Preliminaries}
\label{sec:prep}

We follow \cite{West} for graph-theoretic definition and notation. A \emph{tree} is a connected graph that contains no cycle. In \cite{Daphne1} the \emph{weight} of $T$ from $v \in V(T)$ is defined as $w_{T}(v) = \sum_{u \in V(T)} d(u,v)$ and the \emph{weight} of $T$ as $w(T)$ = min\{$w_{T}(v)$ : $v \in V(T)$\}. A vertex $v \in V(T)$ is a \emph{weight centre} \cite{Daphne1} of $T$ if $w_{T}(v)$ = $w(T)$. Denote by $W(T)$ the set of weight centres of $T$. It was proved in \cite[Lemma 2]{Daphne1} that every tree $T$ has either one or two weight centres, and $T$ has two weight centres, say, $W(T) = \{w, w'\}$, if and only if $w$ and $w'$ are adjacent and $T - ww'$ consists of two equal-sized components. We view $T$ as rooted at its weight centre $W(T)$: if $W(T) = \{w\}$, then $T$ is rooted at $w$; if $W(T) = \{w, w'\}$ (where $w$ and $w'$ are adjacent), then $T$ is rooted at $w$ and $w'$ in the sense that both $w$ and $w'$ are at level $0$. In either case, if in $T$ the unique path from a weight centre to a vertex $v \not \in W(T)$ passes through a vertex $u$ (possibly with $u = v$), then $u$ is called an \emph{ancestor} of $v$, and $v$ is called a \emph{descendent} of $u$. If $v$ is a descendent of $u$ and is adjacent to $u$, then $v$ is a \emph{child} of $u$. Let $u \not \in W(T)$ be adjacent to a weight centre. The subtree induced by $u$ and all its descendent is called a \emph{branch} at $u$. Two branches are called \emph{different} if they are at two vertices adjacent to the same weight centre, and \emph{opposite} if they are at two vertices adjacent to different weight centres. Note that the latter case occurs only when $T$ has two weight centres. Define
$$
L(u) := \min\{d(u, x): x \in W(T)\},\; u \in V(T)
$$
to indicate the \emph{level} of $u$ in $T$. Define the \emph{total level} of $T$ as
$$
L(T) := \sum_{u \in V(T)} L(u).
$$
For any $u, v \in V(T)$, define
$$
\phi(u,v) := \max\{L(x): \mbox{$x$ is a common ancestor of $u$ and $v$}\}
$$
$$
\delta(u,v) := \left\{
\begin{array}{ll}
1, & \mbox{if $W(T) = \{w, w'\}$ and $P_{uv}$ contains the edge $ww'$} \\ [0.2cm]
0, & \mbox{otherwise.}
\end{array}
\right.
$$

\begin{Lemma} \label{obs}
Let $T$ be a tree with diameter $d \ge 2$. Then for any $u, v \in V(T)$ the following hold:
\begin{enumerate}[\rm (a)]
  \item $\phi(u,v) \geq 0$;
  \item $\phi(u,v) = 0$ if and only if $u$ and $v$ are in different or opposite branches;
  \item $\delta(u,v) = 1$ if and only if $T$ has two weight centres and $u$ and $v$ are in opposite branches;
  \item the distance $d(u,v)$ in $T$ between $u$ and $v$ can be expressed as
\be
\label{eq:dist}
d(u,v) = L(u) + L(v) - 2 \phi(u,v) + \delta(u,v).
\ee
\end{enumerate}
\end{Lemma}

\section{Radio number of trees}
\label{sec:tree}

A radio labeling of $T$ is an injective mapping $f$ from $V(T)$ to the set of nonnegative integers; we can always assume that $f$ assigns $0$ to some vertex. Thus $f$ induces a linear order of the vertices of $T$, namely $V(T) = \{u_{0}, u_{1}, \ldots, u_{p-1}\}$ (where $p=|V(T)|$) defined by
$$
0 = f(u_{0}) < f(u_{1}) <  \cdots < f(u_{p-1}) = \span(f).
$$
Define
$$
\ve(T) := \left\{
\begin{array}{ll}
1, & \mbox{if $T$ has only one weight centre} \\ [0.3cm]
0, & \mbox{if $T$ has two (adjacent) weight centres.}
\end{array}
\right.
$$

The following result is essentially the same as \cite[Theorem 3]{Daphne1}, because when $T$ has a unique weight centre, say $w$, we have $L(T) = w_{T}(w) = w(T)$, and when $T$ has two weight centres, say $w$ and $w'$, the number of vertices in each of the two components of $T - ww'$ is equal to $p/2$ (\cite[Lemma 2]{Daphne1}) and $L(T) + p/2 = w_{T}(w) = w_{T}(w') = w(T)$. However, we give a proof of Lemma \ref{thm:lb} as it will be used in the proof of Theorem \ref{thm:ub} and subsequent discussion.

\begin{Lemma}
\label{thm:lb}
(\cite[Theorem 3]{Daphne1})
Let $T$ be a tree with order $p$ and diameter $d \ge 2$. Denote $\ve = \ve(T)$. Then
\be
\label{eq:lb}
\rn(T) \ge (p-1)(d+\ve) - 2 L(T) + \ve.
\ee
\end{Lemma}
\begin{proof}~It suffices to prove that any radio labeling of $T$ has span no less than the right-hand side of (\ref{eq:lb}). Suppose that $f$ is an arbitrary radio labeling of $T$. We order the vertices of $T$ such that $0 = f(u_{0}) < f(u_{1}) < f(u_{2}) < \cdots < f(u_{p-1})$. Since $f$ is a radio labeling, we have $f(u_{i+1}) - f(u_{i}) \geq (d+1)-d(u_{i},u_{i+1})$ for $0 \leq i \leq p-2$. Summing up these $p-1$ inequalities, we obtain
\be
\label{eq:sumup}
\span(f) = f(u_{p-1}) \geq (p-1)(d+1) - \sum_{i = 0}^{p-2} d(u_{i}, u_{i+1}).
\ee

\textsf{Case 1}: $T$ has one weight centre.

In this case, we have $\d(u_{i}, u_{i+1}) = 0$ for $0 \le i \le p-2$ by the definition of the function $\d$. Since $T$ has only one weight centre, $u_{0}$ and $u_{p-1}$ cannot be the root (weight centre) of $T$ simultaneously. Hence $L(u_{0}) + L(u_{p-1}) \ge 1$. Thus, by (\ref{eq:dist}) and Lemma \ref{obs}(a),
\bean
\sum_{i=0}^{p-2} d(u_{i},u_{i+1}) & = & \sum_{i=0}^{p-2} (L(u_{i}) + L(u_{i+1}) - 2\phi(u_{i}, u_{i+1})) \\
& = & 2 L(T) - L(u_{0}) - L(u_{p-1}) - 2 \sum_{i=0}^{p-2}\phi(u_{i}, u_{i+1}) \\
& \le & 2 L(T) - 1.
\eean
This together with (\ref{eq:sumup}) yields $\span(f) = f(u_{p-1}) \geq (p-1)(d+1) - 2 L(T) + 1$.

\textsf{Case 2}: $T$ has two weight centres.

By Lemma \ref{obs}(a), we have $\phi(u_{i}, u_{i+1}) \ge 0$ for $0 \le i \le p-2$. We also have $\d(u_{i}, u_{i+1}) \le 1$ for $0 \le i \le p-2$.
Since $L(u_{0}) \ge 0$ and $L(u_{p-1}) \ge 0$, by (\ref{eq:dist}) we then have
\bean
\sum_{i=0}^{p-2}d(u_{i},u_{i+1}) & = & 2 L(T) - L(u_{0}) - L(u_{p-1}) - 2 \sum_{i=0}^{p-2}\phi(u_{i}, u_{i+1}) + \sum_{i=0}^{p-2} \d(u_{i}, u_{i+1}) \\
& \leq & 2L(T) + (p-1).
\eean
Combining this with (\ref{eq:sumup}) we obtain $\span(f) = f(u_{p-1}) \ge (p-1)d - 2L(T)$.
\qed
\end{proof}

The next result gives a necessary and sufficient condition for the equality in (\ref{eq:lb}) along with an optimal radio labeling. It will be crucial for our subsequent discussion.

\begin{Theorem}
\label{thm:ub}
Let $T$ be a tree with order $p$ and diameter $d \ge 2$. Denote $\ve = \ve(T)$. Then
\be
\label{eq:ub}
\rn(T) = (p-1)(d+\ve) - 2 L(T) + \ve
\ee
holds if and only if there exists a linear order $u_0, u_1, \ldots, u_{p-1}$ of the vertices of $T$ such that
\begin{enumerate}[\rm (a)]
\item $u_0 = w$ and $u_{p-1} \in N(w)$ when $W(T) = \{w\}$, and $\{u_0, u_{p-1}\} = \{w, w'\}$ when $W(T) = \{w, w'\}$;
\item the distance $d(u_{i}, u_{j})$ between $u_i$ and $u_j$ in $T$ satisfies
\be
\label{eq:dij}
d(u_{i}, u_{j}) \ge \sum_{t = i}^{j-1} (L(u_t)+L(u_{t+1})) - (j - i)(d+\ve) + (d+1),\;\, 0 \le i < j \le p-1.
\ee
\end{enumerate}
Moreover, under this condition the mapping $f$ defined by
\be
\label{eq:f0}
f(u_{0}) = 0
\ee
\be
\label{eq:f}
f(u_{i+1}) = f(u_{i}) - L(u_{i+1}) - L(u_{i}) + (d + \ve),\;\, 0 \leq i \leq p-2
\ee
is an optimal radio labeling of $T$.
\end{Theorem}

We need some preparations in order to prove Theorem \ref{thm:ub}. Given a radio labeling $f$ of a tree $T$, define
$$
x_{i} := f(u_{i+1}) - f(u_{i}) + L(u_{i+1}) + L(u_{i}) - (d + \ve),\;\, 0 \leq i \leq p-2,
$$
where $p=|V(G)|$, $d=\diam(T)$ and $\ve = \ve(T)$ as before. Obviously, the values of $x_i$'s rely on $f$.

\begin{Lemma}
\label{lem:xi}
$x_{i} \geq 2 \phi(u_{i},u_{i+1})$ and hence $x_{i} \geq 0$, $0 \le i < p-1$.
\end{Lemma}

\begin{proof}
By Lemma \ref{obs} and the definition of a radio labeling, we have
$x_{i} \geq d + 1 - d(u_{i}, u_{i+1}) + L(u_{i+1}) + L(u_{i}) - (d + \ve)$ = $2 \phi(u_{i}, u_{i+1}) + (1-\ve - \delta(u_{i}, u_{i+1}))$ $\geq 2 \phi(u_{i}, u_{i+1})$.
\qed
\end{proof}

\begin{Lemma}
\label{lem3}
Let $T$ be a tree with order $p$ and diameter $d \ge 2$. Denote $\ve = \ve(T)$.
Let $f$ be an injective mapping from $V(T)$ to the set of nonnegative integers, and let $u_{0}, u_{1}, \ldots, u_{p-1}$ be the vertices of $T$ ordered in such a way that $0 = f(u_{0}) < f(u_{1}) < \cdots < f(u_{p-1})$. Then $f$ is a radio labeling of $T$ if and only if for any $0 \leq i < j \leq p-1$,
\be
\label{eq:sumxi}
\sum_{t = i}^{j - 1} x_{t} \geq
2 \sum_{t = i+1}^{j - 1} L(u_{t}) + 2  \phi(u_{i},u_{j}) - \delta(u_{i},u_{j}) - (j - i)(d+\ve) + (d+1);
\ee
that is,
\be
\label{eq:xi}
\sum_{t = i}^{j - 1} x_{t} \geq
\sum_{t = i}^{j-1} (L(u_t)+L(u_{t+1})) - d(u_{i}, u_{j}) - (j - i)(d+\ve) + (d+1).
\ee
\end{Lemma}

\begin{proof}
We have
\bea
\sum_{t = i}^{j - 1} x_{t} & = & \sum_{t = i}^{j - 1} (f(u_{t+1}) - f(u_{t}) + L(u_{t+1}) + L(u_{t}) - (d+\ve)) \non \\
& = & f(u_{j}) - f(u_{i}) + 2\sum_{t = i+1}^{j - 1} L(u_{t}) + L(u_{i}) + L(u_{j}) - (j-i)(d+\ve). \label{eq:sum}
\eea
Thus, if $f$ is a radio labeling of $T$, then by (\ref{eq:dist}), for any $i, j$ with $0 \leq i < j \leq p-1$,
\bean
\sum_{t = i}^{j - 1} x_{t} & \ge & d + 1 - d(u_{i}, u_{j}) + 2 \sum_{t = i+1}^{j - 1} L(u_{t}) + L(u_{i}) + L(u_{j}) - (j-i)(d+\ve) \\
& = & 2 \sum_{t = i+1}^{j - 1} L(u_{t}) + 2 \phi(u_{i},u_{j}) - \delta(u_{i}, u_{j}) - (j - i)(d+\ve) + (d+1).
\eean

Conversely, if $f$ satisfies (\ref{eq:sumxi}), then by (\ref{eq:sum}) and (\ref{eq:dist}) and the definition of $x_t$,
\bean
f(u_{j}) - f(u_{i}) & = & \sum_{t = i}^{j-1} (f(u_{t+1}) - f(u_{t})) \\
& = & \sum_{t = i}^{j-1} (x_{t} - L(u_{t}) - L(u_{t+1}) + (d+\ve)) \\
& = & \sum_{t = i}^{j-1} x_{t} - 2 \sum_{t = i+1}^{j-1} L(u_t) - L(u_{i}) - L(u_{j}) + (j - i)(d+\ve)\\
& \geq & d + 1 - (L(u_{i}) + L(u_{j}) - 2 \phi(u_{i},u_{j}) + \delta(u_{i}, u_{j})) \\
& = & d + 1 - d(u_{i}, u_{j})
\eean
and hence $f$ is a radio labeling of $T$.
\qed
\end{proof}

\begin{proof}\textbf{of Theorem \ref{thm:ub}}~
\textsf{Necessity}:~Suppose that (\ref{eq:ub}) holds. Let $f$ be an optimal labeling of $T$ with the corresponding ordering of vertices given by $0 = f(u_{0}) < f(u_{1}) < f(u_{2}) < \cdots < f(u_{p-1})$. Then $\span(f) = \rn(T) = (p-1)(d+\ve) - 2 L(T) + \ve$. Thus from the proof of Lemma \ref{thm:lb} all inequalities there for $f$ must be equalities. More explicitly, we have $f(u_{i+1}) - f(u_{i}) = (d+1)-d(u_{i},u_{i+1})$ for $0 \leq i \leq p-2$, and (i) if $T$ has a unique weight centre, then $L(u_{0}) + L(u_{p-1}) =1$ and $\phi(u_{i}, u_{i+1}) = 0$ for $0 \le i \le p-2$, and (ii) if $T$ has two weight centres, then $L(u_{0}) = L(u_{p-1}) = 0$ (that is, $\{u_0, u_{p-1}\} = \{w, w'\}$) and $\phi(u_{i}, u_{i+1}) = 0$, $\d(u_{i}, u_{i+1}) = 1$ for $0 \le i \le p-2$. In the former case, we may assume without loss of generality that $L(u_{0}) = 0$ and $L(u_{p-1}) = 1$ (that is, $u_0 = w$ and $u_{p-1}$ is adjacent to $w$), because the mapping $\span(f) - f$ is also an optimal radio labeling of $T$. In either case, by (\ref{eq:dist}), we have $f(u_{i+1}) - f(u_{i}) = (d+\ve)-L(u_{i+1})-L(u_{i})$, that is, $x_i = 0$, for $0 \leq i \leq p-2$. Since $f$ is a radio labeling, it satisfies (\ref{eq:xi}). So the right-hand side of (\ref{eq:xi}) must be non-positive and (\ref{eq:dij}) follows.

\textsf{Sufficiency}:~Suppose that a linear order $u_0, u_1, \ldots, u_{p-1}$ of the vertices of $T$ satisfies (\ref{eq:dij}), and $f$ is defined by (\ref{eq:f0}) and (\ref{eq:f}). By Lemma \ref{thm:lb} it suffices to prove that $f$ is a radio labeling of $T$ and $\span(f) = (p-1)(d+\ve) - 2 L(T) + \ve$.

In fact, since $T$ has diameter $d$, we have $L(u_{i}) + L(u_{i+1}) < d + \ve$ for $0 \leq i \leq p-2$. Thus, by (\ref{eq:f0}) and (\ref{eq:f}), we have $0 = f(u_{0}) < f(u_{1}) < \cdots < f(u_{p-1})$. By (\ref{eq:f}), we have $x_i = 0$ for $0 \le i \le p-2$. Thus, for $0 \le i < j \le p-1$, the left-hand side of (\ref{eq:xi}) is equal to $0$. On the other hand, by (\ref{eq:dij}), the right-hand side of the same equation is non-positive. Therefore, $f$ satisfies (\ref{eq:xi}) and so is a radio labeling of $T$. The span of $f$ is given by
\bean
\span(f) & = & f(u_{p-1}) - f(u_{0}) \\
& = & \sum_{i=0}^{p-2} (f(u_{i+1}) - f(u_{i})) \\
& = & (p-1)(d+\ve) - \sum_{i=0}^{p-2} (L(u_{i+1}) + L(u_{i})) \\
& = & (p-1)(d+\ve) - 2 L(T) + L(u_0) + L(u_{p-1}) \\
& = & (p-1)(d+\ve) - 2 L(T) + \ve.
\eean
Therefore, $\rn(T) \le (p-1)(d+\ve) - 2 L(T) + \ve$. This together with (\ref{eq:lb}) implies (\ref{eq:ub}) and that $f$ is an optimal radio labeling of $T$.
\qed
\end{proof}

\begin{Remark}
{\em In general, it seems difficult to decide whether a general tree satisfies the conditions in Theorem \ref{thm:ub}, and if it satisfies these conditions how we can find the linear order meeting (a)-(b) in Theorem \ref{thm:ub}. Nevertheless, these may be achieved  for some special families of trees such as the ones in the next section.
	
Consider the following properties:
	
($A_i$) $u_{i}$ and $u_{i+1}$ are in different branches when $W(T) = \{w\}$ and in opposite branches when $W(T) = \{w, w'\}$;

($B_i$) $L(u_{i}) \leq (d+1)/2$ when $W(T) = \{w\}$, and $L(u_{i}) \leq (d-1)/2$ when $W(T) = \{w, w'\}$;

($C_{ij}$) $\phi(u_{i},u_{j}) \leq (j-i-1)((d+\ve)/2) - \sum_{t=i+1}^{j-1}L(u_{t})-((1-\ve)/2)$ if $u_{i}$ and $u_{j}$ are in the same branch.

It can be verified that any linear order $u_{0}, u_{1}, \ldots, u_{p-1}$ that satisfies (a)-(b) in Theorem \ref{thm:ub} also satisfies ($A_i$) for $0 \le i \le p-2$, ($B_i$) for $0 \leq i \leq p-1$ and ($C_{ij}$) for $1 \leq i < j \leq p-1$. In fact, one can show ($A_i$) by taking $j = i+1$ in \eqref{eq:dij} and using Lemma \ref{obs}. Clearly, ($B_0$) and ($B_{p-1}$) hold. For $1 \le i \le p-2$, by applying (\ref{eq:dij}) to $d(u_{i-1}, u_{i+1})$ and noting $\phi(u_{i-1},u_{i+1}) \geq 0$, we obtain $L(u_i) \le L(u_i) + \phi(u_{i-1},u_{i+1}) \leq (d+1)/2$ if $W(T) = \{w\}$ and $L(u_i) \le L(u_i) + \phi(u_{i-1},u_{i+1}) \leq (d-1)/2$ if $W(T) = \{w, w'\}$, and hence ($B_i$) holds. Using \eqref{eq:dist} and \eqref{eq:dij}, one can show that ($C_{ij}$) holds for $1 \leq i < j \leq p-1$.

Inspired by the properties above, one may try to test whether the vertices of a given tree $T$ can be ordered in such a way that (a) and (b) in Theorem \ref{thm:ub} are satisfied, and to produce such a linear order if it exists, by using the following procedure:

(i) Set $u_0 = w$.

(ii) Choose $u_{p-1} \in N(w)$ if $W(T) = \{w\}$ and $u_{p-1} = w'$ if $W(T) = \{w, w'\}$.

(iii) Choose a vertex $u_1$ (in any branch) other than $u_0$ and $u_{p-1}$ such that property ($B_1$) is respected.

(iv) In general, suppose that $u_0, u_1, \ldots, u_{t}$ have been put in order such that ($A_i$) holds for $0 \le i \le t-1$, ($B_i$) holds for $0 \leq i \leq t$, and ($C_{ij}$) holds for $1 \leq i < j \leq t$. If $t = p-2$, stop and output the linear order $u_0, u_1, \ldots, u_{p-1}$. If $t < p-2$, choose $u_{t+1}$ from $V(T) \setminus \{u_0, u_1, \ldots, u_{t}, u_{p-1}\}$ such that ($A_t$), ($B_{t+1}$) and ($C_{i, t+1}$), $1 \leq i < t+1$ are respected, and continue the process with the longer sequence $u_0, u_1, \ldots, u_{t}, u_{t+1}$, if such a vertex $u_{t+1}$ exists. (It can be verified that ($C_{i, t+1}$) holds if $L(u_{t+1})+L(u_{i}) < \sum_{k=i+1}^{t}(d+\ve-2L(u_{k}))+\ve$. So it suffices to ensure that ($C_{i, t+1}$) is respected for those $i$ such that $L(u_{t+1})+L(u_{i}) \ge \sum_{k=i+1}^{t}(d+\ve-2L(u_{k}))+\ve$.) If no such a vertex $u_{t+1}$ exists, one may try to choose a different $u_i$ for some $i \le t$ and run the procedure for the sequence $u_0, u_1, \ldots, u_{i}$.

It can be proved that, if we terminate with $t=p-2$, then the linear order $u_0, u_1, \ldots, u_{p-1}$ produced above satisfies (a) and (b) in Theorem \ref{thm:ub} and therefore the radio number of $T$ is given by \eqref{eq:ub}.

Obviously, the procedure above is not an algorithm, but it can be easily modified to give an enumerative algorithm by considering all possible choices for $u_{t+1}$ in each iteration. Though this algorithm is likely to be exponential in general (as it requires enumeration of a large number of possibilities), it may be efficient for some families of trees with special structures or small orders.
}
\end{Remark}

We now present two sets of sufficient conditions for \eqref{eq:ub} to hold. These conditions are easier to verify than \eqref{eq:dij} in some cases and will be used in the next section.

\begin{Theorem}
\label{thm:cor1}
Let $T$ be a tree with order $p$ and diameter $d \ge 2$. Denote $\ve = \ve(T)$. Suppose that there exists a linear order $u_0, u_1, \ldots, u_{p-1}$ of the vertices of $T$ such that
\begin{enumerate}[\rm (a)]
\item $u_0 = w$ and $u_{p-1} \in N(w)$ when $W(T) = \{w\}$, and $\{u_0, u_{p-1}\} = \{w, w'\}$ when $W(T) = \{w, w'\}$;
\item $u_{i}$ and $u_{i+1}$ are in different branches if $W(T) = \{w\}$ and in opposite branches if $W(T) = \{w, w'\}$, $0 \le i \le p-2$;
\end{enumerate}
and one of the following holds:
\begin{enumerate}[\rm (c)]
\item  $\min\{d(u_{i},u_{i+1}),d(u_{i+1},u_{i+2})\} \leq (d+1-\ve)/2,\;  0 \leq i \leq p-3$;
\item[\rm (d)]  $d(u_{i},u_{i+1}) \leq (d+1+\ve)/2$, $0 \leq i \leq p-2$.
\end{enumerate}
Then $\rn(T)$ is given by \eqref{eq:ub} and $f$ defined in \eqref{eq:f0}-\eqref{eq:f} is an optimal radio labeling of $T$.
\end{Theorem}

\begin{proof}
It suffices to prove that the linear order $u_0, u_1, \ldots, u_{p-1}$ satisfies (\ref{eq:dij}). Denote by $S_{i,j}$ the right-hand side of (\ref{eq:dij}) with respect to this order. In view of \eqref{eq:dist}, we may assume $j-i \geq 2$.

\textsf{Case 1:} $W(T) = \{w\}$.

Suppose (a), (b) and (c) hold. If $j \geq i+4$, then $\min \{L(u_{t})+L(u_{t+1}): i \le t \le j-1\} \leq d/2$ and $\max  \{L(u_{t})+L(u_{t+1}): i \le t \le j-1\} = d$ as $\min\{d(u_{i},u_{i+1}), d(u_{i+1},u_{i+2})\} \leq d/2$. Hence $S_{i,j} \leq \((j-i)/2\)\(d/2+d\)-3(d+1) \leq 2\(3d/2\)-3d-3 = -3 < d(u_{i},u_{j})$. If $j = i+3$, then either (i) $d(u_{i},u_{i+1}) \leq d/2$, $d(u_{i+1},u_{i+2}) > d/2$ and $d(u_{i+2},u_{i+3}) \leq d/2$, or (ii) $d(u_{i},u_{i+1}) > d/2$, $d(u_{i+1},u_{i+2}) \leq d/2$ and $d(u_{i+2},u_{i+3}) > d/2$. In case (i), we have $L(u_{i})+L(u_{i+1}) \leq d/2$, $d/2 < L(u_{i+1})+L(u_{i+2}) \leq d$ and $L(u_{i+2})+L(u_{i+3}) \leq d/2$. Hence $S_{i,j} \leq \(d/2+d+d/2\)-2(d+1) = -2 < d(u_{i},u_{j})$. In case (ii), we have $d/2 < L(u_{i})+L(u_{i+1}) < d$, $L(u_{i+1})+L(u_{i+2}) \leq d/2$, $d/2 < L(u_{i+2})+L(u_{i+3}) \leq d$, and $d(u_{i},u_{i+3}) \geq d/2$ as $u_{i}$ and $u_{i+3}$ are in different branches. Hence $S_{i,j} \leq \(d+d/2+d\)-2(d+1) = d/2-2 < d(u_{i},u_{j})$. If $j = i+2$, then either (i) $d(u_{i},u_{i+1}) \leq d/2$ and $d(u_{i+1},u_{i+2}) \leq d/2$,  or (ii) $d(u_{i},u_{i+1}) \leq d/2$ and $d/2 < d(u_{i+1},u_{i+2}) \leq d$. In case (i), we have $L(u_{i})+L(u_{i+1}) \leq d/2$ and $L(u_{i+1})+L(u_{i+2}) \leq d/2$, and hence $S_{i,j} \leq \(d/2+d/2\)-(d+1) = -1 < d(u_{i},u_{j})$. In case (ii), we have $L(u_{i})+L(u_{i+1}) \leq d/2$ and $L(u_{i+1})+L(u_{i+2}) = d(u_{i+1},u_{i+2})$, and hence $S_{i,j} \leq \(d/2+d(u_{i+1},u_{i+2})\)-(d+1) = d(u_{i+1},u_{i+2})-\(d/2-1\) < d(u_{i},u_{j})$.

Suppose (a), (b) and (d) hold. Then, for $0 \leq i \leq p-2$, $d(u_{i},u_{i+1}) \leq (d+1+\ve)/2$ and $L(u_{i})+L(u_{i+2}) \leq (d+2)/2$ as $u_{i}$ and $u_{i+1}$ are in different branches. Hence $S_{i,j} \leq \sum_{t=i}^{j-1}\left((d+2)/2\right)-(j-i)(d+1)+(d+1) = (d+1)-(j-i)\left(d/2\right) \leq 1 \leq d(u_{i},u_{i+1})$ as $j-i \geq 2$.

\textsf{Case 2:} $W(T) = \{w, w'\}$.

Suppose (a), (b) and (c) hold. If $j \geq i+4$, then $\min\{L(u_{t})+L(u_{t+1}): i \le t \le j-1\} \leq (d-1)/2$ and $\max\{L(u_{t})+L(u_{t+1}): i \le t \le j-1 \} = d-1$ as $\min\{d(u_{i},u_{i+1}),d(u_{i+1},u_{i+2})\} \leq (d+1/2)$. Hence $S_{i,j} \leq \((j-i)/2\)\((d-1)/2+d-1\)-4d+(d+1) \leq 2\(3(d-1)/2\)-3d+1 = -2 < d(u_{i},u_{j})$. If $j = i+3$, then either (i) $d(u_{i},u_{i+1}) \leq (d+1)/2$, $d/2 < d(u_{i+1},u_{i+2}) \leq d$ and $d(u_{i+2},u_{i+3}) \leq (d+1)/2$, or (ii) $(d+1)/2 < d(u_{i},u_{i+1}) \leq d$, $d(u_{i+1},u_{i+2}) \leq d/2$ and $(d+1)/2 < d(u_{i+2},u_{i+3}) \leq d$. In case (i), $L(u_{i})+L(u_{i+1}) \leq (d-1)/2$, $(d-1)/2 < L(u_{i+1})+L(u_{i+2}) \leq d-1$ and $L(u_{i+2})+L(u_{i+3}) \leq (d-1)/2$. Hence $S_{i,j} \leq \((d-1)/2+d-1+(d-1)/2\)-3d+(d+1) \leq \(2d-2\)-2d+1 = -1 < d(u_{i},u_{j})$. In case (ii), $(d+1)/2 < d(u_{i},u_{i+1}) \leq d$, $d(u_{i+1},u_{i+2}) \leq (d+1)/2$ and $(d+1)/2 < d(u_{i+2},u_{i+3}) \leq d$. Hence $S_{i,j} \leq \(d-1+(d-1)/2+d-1\)-3d+(d+1) \leq \(5(d-1)/2\)-2d+1 = (d-3/2) < d(u_{i},u_{j})$ as $u_{i}$ and $u_{j}$ are in opposite branches. If $j = i+2$, then either (i) $d(u_{i},u_{i+1}) \leq (d+1)/2$ and $d(u_{i+1},u_{i+2}) \leq (d+1)/2$, or (ii) $d(u_{i},u_{i+1}) \leq (d+1)/2$ and $d(u_{i+1},u_{i+2}) > (d+1)/2$. In the former case, we have $L(u_{i})+L(u_{i+1}) \leq (d-1)/2$ and $L(u_{i+1})+L(u_{i+2}) \leq (d-1)/2$, and hence $S_{i,j} \leq \((d-1)/2+(d-1)/2\)-2d+(d+1) \leq \(d-1\)-d+1 = 0 < d(u_{i},u_{j})$. In the latter case, we have $L(u_{i})+L(u_{i+1}) \leq (d-1)/2$ and $(d-1)/2 < L(u_{i+1})+L(u_{i+2}) \leq d-1$, and hence $S_{i,j} \leq \((d-1)/2 + d(u_{i+1},u_{i+2}) - 1\)-2d+(d+1) \leq d(u_{i+1},u_{i+2}) - \((d+1)/2\) < d(u_{i},u_{j})$.

Suppose (a), (b) and (d) hold. Then, for $0 \leq i \leq p-2$, $d(u_{i},u_{i+1}) \leq (d+1+\ve)/2$ and $L(u_{i})+L(u_{i+1}) \leq (d-1)/2$ as $u_{i}$ and $u_{i+1}$ are in opposite branches. Hence $S_{i,j} \leq \sum_{t=i}^{j-1}\left((d-1)/2\right)-(j-i) d + (d+1) = (d+1)-(j-i)\left((d+1)/2\right) \leq 0 < d(u_{i},u_{i+1})$ as $j-i \geq 2$.
\qed
\end{proof}

The proof of Theorems \ref{thm:ub} and \ref{thm:cor1} implies the following result which will be used in the next section.

\begin{Theorem}
\label{thm:ub1}
Let $T$ be a tree with order $p$ and diameter $d \ge 2$. Denote $\ve = \ve(T)$. Then, for any linear order $u_0, u_1, \ldots, u_{p-1}$ of the vertices of $T$ satisfying (\ref{eq:dij}), or (b) and one of (c) and (d) in Theorem \ref{thm:cor1}, the mapping $f$ given by (\ref{eq:f0}) and (\ref{eq:f}) is a radio labeling of $T$. Moreover, if in addition $L(u_0) + L(u_{p-1}) = k + 1$ when $W(T) = \{w\}$ and $L(u_0) + L(u_{p-1}) = k$ when $W(T) = \{w, w'\}$, then
\be
\label{eq:ub1}
\rn(T) \le (p-1)(d+\ve) - 2 L(T) + \ve + k
\ee
and \span(f) is equal to this upper bound.
\end{Theorem}

\section{Radio number for three families of trees}
\label{sec:three}

In this section we use Theorems \ref{thm:ub}, \ref{thm:cor1} and \ref{thm:ub1} to determine the radio number for three families of trees. We continue to use the terminology and notation in the previous section.

\subsection{Banana trees}

A $k$-star is a tree consisting of $k$ leaves and another vertex joined to all leaves by edges.
We define the \emph{$(n,k)$-banana tree}, denoted by $B(n,k)$, to be the tree obtained by joining one leaf of each of $n$ copies of a $(k-1)$-star to a single root (which is distinct from all vertices in the $k$-stars). See Fig. \ref{fig1} for an illustration. It is clear that $B(n,k)$ has diameter 6 and exactly one weight centre if $n \ge 2$.

\begin{Theorem}
\label{thm:banana}
Let $n \geq 5$ and $k \geq 4$ be integers. Then
$$
\rn(B(n,k)) = n(k+6)+1.
$$
\end{Theorem}

\begin{proof}
The order and total level of $B(n,k)$ are given by $p = nk + 1$ and $L(B(n,k)) = 3n(k-1)$ respectively. Plugging these into (\ref{eq:lb}), we obtain $\rn(B(n,k)) \ge n(k+6)+1$. We now prove that this lower bound is tight by giving a linear order of the vertices of $B(n,k)$ satisfying (\ref{eq:dij}).

Let $w^{i}_{1}, w^{i}_{2}, \ldots, w^{i}_{k}$ denote the vertices of the $i^{th}$ copy of the $(k-1)$-star in $B(n,k)$, where $w^{i}_{1}$ is the apex vertex (centre) and $w^{i}_{2}, \ldots, w^{i}_{k}$ are the leaves. Without loss of generality we assume that $w^{1}_{k}, w^{2}_{k}, \ldots, w^{n}_{k}$ are joined by edges to a common vertex $w$, which is the unique weight centre of $B(n,k)$.

We give a linear order $u_{0}, u_{1}, u_{2}, \ldots, u_{p-1}$ of the vertices of $B(n,k)$ as follows. We first set $u_{0}$ = $w$. Next, for $1 \leq t \leq p-1$, let
\begin{center}
$u_{t}$ := $w^{i}_{j}$, where $t$ = $(j-1)n + i$, $1 \leq i \leq n$, $1 \leq j \leq k$.
\end{center}
Note that $u_{p-1}$ = $w_{k}^{n}$ is adjacent to $w$ and for $1 \leq i \leq p-2$, $u_{i}$ and $u_{i+1}$ are in different branches so that $\phi(u_{i},u_{i+1})$ = 0.

\textsf{Claim:} The linear order $u_{0}, u_{1}, \ldots, u_{p-1}$ above satisfies (\ref{eq:dij}).

To prove this consider any two vertices $u_{i}, u_{j}$ of $B(n,k)$ with $0 \leq i < j \leq p-1$. Since the diameter of $B(n,k)$ is $6$, the right-hand side of (\ref{eq:dij}) is given by $S_{i,j} := \sum_{t=i}^{j-1}(L(u_{t})+L(u_{t+1})-7)+7$. It is easy to verify that (\ref{eq:dij}) holds when $i=0$. We assume $i \ge 1$ in the sequel.

\textsf{Case 1:} $1 \leq i < j \leq n$.~~We have $d(u_{i},u_{j}) = 4$ and $L(u_{t}) = 2$ for $i \leq t \leq j$. Hence $S_{i,j} = 7-3(j-i) \leq 4 = d(u_{i},u_{j})$.

\textsf{Case 2:} $i \leq n < j$.~~We have $L(u_{i}) = 2$ and $L(u_{t}) \leq 3$ for $i < t \leq j$ and hence $S_{i,j} \leq 5-(j-i-1)$. If $j-i < n$, then $S_{i,j} \leq 5 = d(u_{i},u_{j})$. If $j-i \geq n$, then $S_{i,j} \leq 1 \le d(u_{i},u_{j})$.

\textsf{Case 3:} $n < i < j \leq p-n-1$.~~We have $L(u_{t})$ = 3 for $i \leq t \leq j$ and hence $S_{i,j}$ = $7-(j-i)$. If $j-i < n$, then $S_{i,j} \leq 6 = d(u_{i},u_{j})$. If $j-i \geq n$, then $S_{i,j} \leq 2 \le d(u_{i},u_{j})$.

\textsf{Case 4:} $i \leq p-n-1 < j$.~~We have $L(u_{t}) \leq 3$ for $i \leq t \leq j-1$ and $L(u_{j}) = 1$, and so $S_{i,j} \leq 4-(j-i-1)$. If $j-i < n$, then $S_{i,j} \leq 4 = d(u_{i},u_{j})$. If $j-i \geq n$, then $S_{i,j} \leq 0 < 2 \le d(u_{i},u_{j})$.

\textsf{Case 5:} $p-n-1 < i < j \leq p-1$.~~We have $L(u_{t}) = 1$ for $i \leq t \leq j$. Hence $S_{i,j} = 7-5(j-i) \leq 2 = d(u_{i},u_{j})$.

So far we have proved the claim above. Therefore, by Theorem \ref{thm:ub}, we have $\rn(B(n,k)) = n(k+6)+1$ and moreover the labeling given by (\ref{eq:f0})-(\ref{eq:f}) (applied to the current case) is an optimal radio labeling of $B(n,k)$.
\qed
\end{proof}

The reader is referred to Fig. \ref{fig1} for an illustration of naming, ordering and labeling of the vertices of $B(5,4)$ by using the procedure in the proof of Theorem \ref{thm:banana}.
\begin{figure}[ht]
\centering
\includegraphics*[scale=0.3]{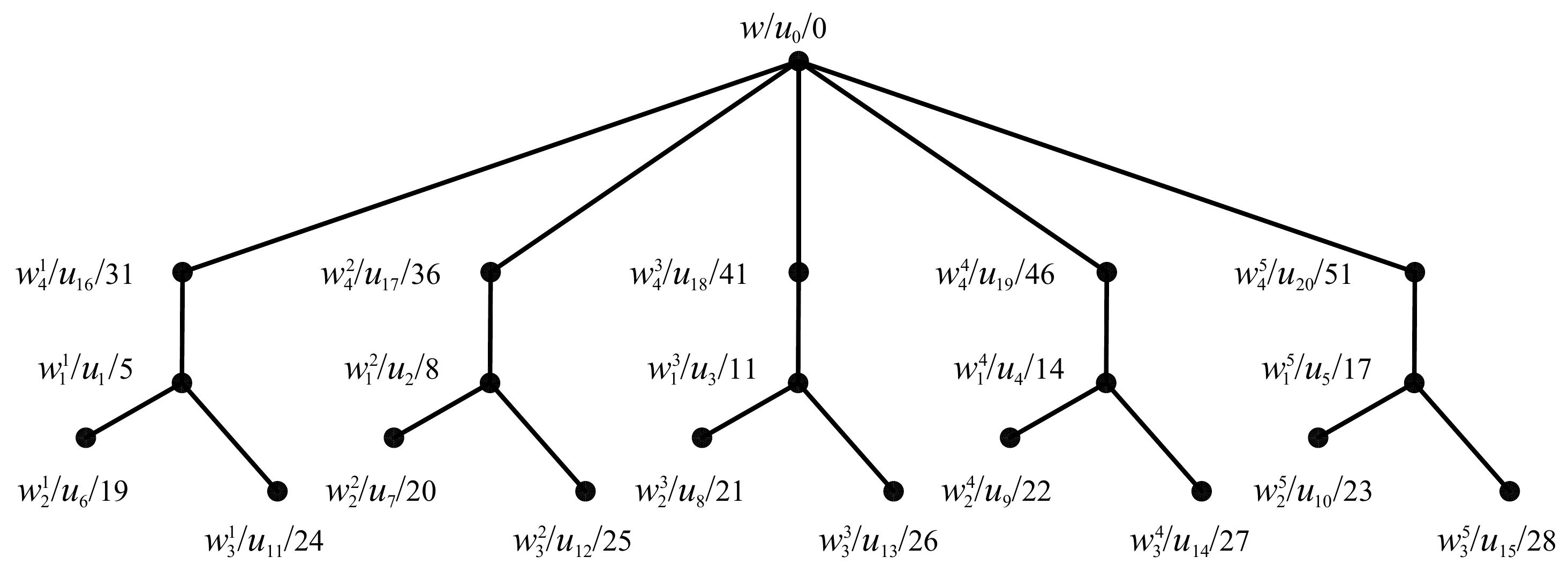}
\caption{\small An optimal radio labeling of $B(5,4)$ together with the corresponding ordering of vertices.}
\label{fig1}
\end{figure}

\subsection{Firecracker trees}

We define the \emph{$(n,k)$-firecracker tree}, denoted by $F(n,k)$, to be the tree obtained by taking $n$ copies of a $(k-1)$-star and identifying a leaf of each of them to a different vertex of a path of length $n-1$ (see Fig. \ref{fig2A}-\ref{fig2B}). It is clear that $F(n,k)$ has one or two weight centres depending on whether $n$ is odd or even.

\begin{Theorem}
\label{thm:fire}
Let $n, k \geq 3$ be integers. Denote $\ve = \ve(F(n,k))$, which is $1$ if $n$ is odd and $0$ if $n$ is even. Then
\begin{equation}
\label{fire:rn}
\rn(F(n,k)) = \frac{(n^{2}+\ve)k}{2}+5n-3.
\end{equation}
\end{Theorem}

\begin{proof}
$F(n,k)$ has order $p = nk$, diameter $d = n+3$ and total level
$$
L(F(n,k)) =  \left\{
\begin{array}{ll}
\frac{1}{4}\left(kn^{2}+(8k-12)n-k\right), & \mbox{if $n$ is odd} \\ [0.3cm]
\frac{1}{4}\left(kn^{2}+6n(k-2)\right), & \mbox{if $n$ is even}.
\end{array}
\right.
$$
Plugging these into (\ref{eq:lb}), we obtain that the right-hand side of (\ref{fire:rn}) is a lower bound for $\rn(F(n,k))$. In what follows we prove that this lower bound is tight by giving a linear order $u_{0}, u_1, u_{2}, \ldots, u_{p-1}$ of the vertices of $F(n,k)$ satisfying (\ref{eq:dij}).

Let $w^{i}_{1}, w^{i}_{2}, \ldots, w^{i}_{k}$ denote the vertices of the $i^{th}$ copy of the $(k-1)$-star in $F(n,k)$, where $w^{i}_{1}$ is the apex vertex (centre) and $w^{i}_{2}, \ldots, w^{i}_{k}$ are the leaves. Without loss of generality we assume that $w^{1}_{k}, w^{2}_{k}, \ldots, w^{n}_{k}$ are identified to the vertices in the path of length $n-1$ in the definition of $F(n,k)$.

\textsf{Case 1}: $n$ is odd.~~In this case, $F(n,k)$ has only one weight centre, namely $w = w_{k}^{(n+1)/2}$. Set $u_{0} = w$. For $1 \leq t \leq p-n$, let
\begin{equation*}
u_{t} := w^{i}_{j}, \mbox{ where  } t = \left\{
\begin{array}{ll}
jn, & \mbox{if } i = (n+1)/2\\ [0.3cm]
(j-1)n + 2i-1, & \mbox{if } i < (n+1)/2 \\ [0.3cm]
(j-1)n + 2\left(i - \frac{n+1}{2}\right), & \mbox{if } i > (n+1)/2.
\end{array}
\right.
\end{equation*}
For $p-n+1 \leq t \leq p-1$, let
\begin{equation*}
u_{t} := w^{i}_{j}, \mbox{ where  } t = \left\{
\begin{array}{ll}
(j-1)n - 2\left(i - \frac{n-1}{2}\right)+1, & \mbox{if } i < (n+1)/2 \\ [0.3cm]
(j-1)n + 2(n-i+1), & \mbox{if } i > (n+1)/2.
\end{array}
\right.
\end{equation*}
Note that $u_{p-1} = w_{k}^{(n+3)/2}$ is adjacent to $w$ and for $1 \leq i \leq p-2$, $u_{i}$ and $u_{i+1}$ are in different branches so that $\phi(u_{i},u_{i+1})$ = 0.

\textsf{Claim 1:} The linear order $u_{0}, u_{1}, \ldots, u_{p-1}$ above satisfies (\ref{eq:dij}).

In fact, for any two vertices $u_{i}, u_{j}$ with $0 \leq i < j \leq p-1$, the right-hand side of (\ref{eq:dij}) is $S_{i,j} := \sum_{t=i}^{j-1} \left(L(u_{t})+L(u_{t+1})-(d+1)\right)+(d+1)$, where $d=n+3 \ge 6$ is the diameter of $F(n,k)$. If $j$ = $i+1$ or $i = 0$, then it is straightforward to verify that (\ref{eq:dij}) is satisfied. Note that, if $j \geq i+3$, then for at least one $t$, $L(u_{t})+L(u_{t+1}) \leq (d+4)/2$ and $L(u_{t})+L(u_{t+1}) \leq (d+6)/2$ for all $i \le t \le j-1$; hence $S_{i,j} \le (j-i-1) ((d+6)/2 - (d+1)) + (d+4)/2 \le 2 ((d+6)/2 - (d+1)) + (d+4)/2 = (-d+12)/2 \le d(u_{i},u_{j})$ as $d \ge 6$. Thus (\ref{eq:dij}) is satisfied when $j \geq i+3$. It remains to consider the case $j = i+2 \ge 3$. In this case we have $S_{i,j} = L(u_{i})+2L(u_{i+1})+L(u_{i+2})-(d+1)$ and one can verify the following: If $1 \leq i \leq n-2$, then $S_{i,j} = 0 < d(u_{i},u_{j}) = 3$; if $n-1 \le i \leq n$, then $S_{i,j} \leq 3 < d(u_{i},u_{j})$; if $n+1 \le i \leq p-n-2$, then $S_{i,j} \leq 4 < d(u_{i},u_{j})$; if $i = p-n-1$, then $S_{i,j} = (-d+8)/2 \le 1 < d(u_{i},u_{j})$; if $i = p-n$, then $S_{i,j} = (-d+2)/2 < 0 < d(u_{i},u_{j})$; if $p-n+1 \le i \leq p-3$, then $S_{i,j} = -2 < d(u_{i},u_{j})$. This completes the proof of the claim.

\textsf{Case 2}: $n$ is even.~~In this case, $F(n,k)$ has two weight centres, namely $w = w_{k}^{n/2}$ and $w' = w_{k}^{n/2+1}$. We set $u_{0} = w'$ and $u_{p-1} = w$. For $1 \leq t \leq p-n+1$, let
\begin{equation*}
u_{t} := w^{i}_{j}, \mbox{ where  } t = \left\{
\begin{array}{ll}
(j-1)n + 2i - 1, & \mbox{if } i \leq n/2 \\ [0.3cm]
(j-1)n + 2\left(i-\frac{n}{2}\right), & \mbox{if } i > n/2.
\end{array}
\right.
\end{equation*}
For $p-n+2 \leq t \leq p-2$, let
\begin{equation*}
u_{t} := w^{i}_{j}, \mbox{ where  } t = \left\{
\begin{array}{ll}
(j-1)n + 2i - 1, & \mbox{if } i < n/2 \\ [0.3cm]
(j-1)n + 2\left(i-1-\frac{n}{2}\right), & \mbox{if } i > n/2 + 1.
\end{array}
\right.
\end{equation*}

Note that $u_{p-1} = w_{k}^{n/2}$ is adjacent to $w'$ and for $1 \leq i \leq p-2$, $u_{i}$ and $u_{i+1}$ are in opposite branches so that $\phi(u_{i},u_{i+1})$ = 0 and $\delta(u_{i},u_{i+1})$ = 1.

\textsf{Claim 2:} The linear order $u_{0}, u_{1}, \ldots, u_{p-1}$ above satisfies (\ref{eq:dij}).

In fact, for any two vertices $u_{i}, u_{j}$ with $0 \leq i < j \leq p-1$, the right-hand side of (\ref{eq:dij}) is $S_{i,j} := \sum_{t=i}^{j-1} \left(L(u_{t})+L(u_{t+1})-d\right)+(d+1)$, where $d=n+3 \ge 6$. If $j$ = $i+1$ or $i = 0$, then it is easy to verify that (\ref{eq:dij}) is satisfied. Note that, if $j \geq i+3$, then $L(u_{t})+L(u_{t+1}) \leq (d+3)/2$ for $i \le t \le j-1$ and hence $S_{i,j} \le (j-i) ((d+3)/2 - d) + (d+1) \le 3 ((d+3)/2 - d) + (d+1) = -(d+11)/2 < d(u_{i},u_{j})$. It remains to consider the case $j = i+2 \ge 3$. In this case $S_{i,j} = L(u_{i})+2L(u_{i+1})+L(u_{i+2})-d+1$ and the following hold: If $1 \leq i \leq n-2$, then $S_{i,j} = -1 < 3 = d(u_{i},u_{j})$; if $n-1 \le i \leq n$, then $S_{i,j}  \leq (3d-1)/2 -d+1 = (d+1)/2 \le d(u_{i},u_{j})$; if $n+1 \le i \leq p-n-2$, then $S_{i,j} = 3 < 5 \le d(u_{i},u_{j})$; if $i = p-n-1$, then $S_{i,j} = (d-1)/2 = d(u_{i},u_{j})$; if $i = p-n$, then $S_{i,j} = (d-7)/2 < (d-3)/2 = d(u_{i},u_{j})$; if $p-n+1 \le i \leq p-3$, then $S_{i,j} = -3 < d(u_{i},u_{j})$. This completes the proof of Claim 2.

In summary, in each case above we have defined a linear order $u_{0}, u_{1}, \ldots, u_{p-1}$ of the vertices of $F(n,k)$ which satisfies (\ref{eq:dij}). Therefore, by Theorem \ref{thm:ub}, $\rn(F(n,k))$ is given by (\ref{eq:ub}) which is exactly the right-hand side of (\ref{fire:rn}) in the case of firecracker trees. The labeling given by (\ref{eq:f0})-(\ref{eq:f}) (applied to the current situation) is an optimal radio labeling of $F(n,k)$.
\qed
\end{proof}

Fig. \ref{fig2A} and \ref{fig2B} give illustrations of naming, ordering and labeling of the vertices of $F(5,4)$ and $F(6,4)$, respectively, by using the procedure in the proof of Theorem \ref{thm:fire}.

\begin{figure}[ht]
\centering
\includegraphics*[scale=0.3]{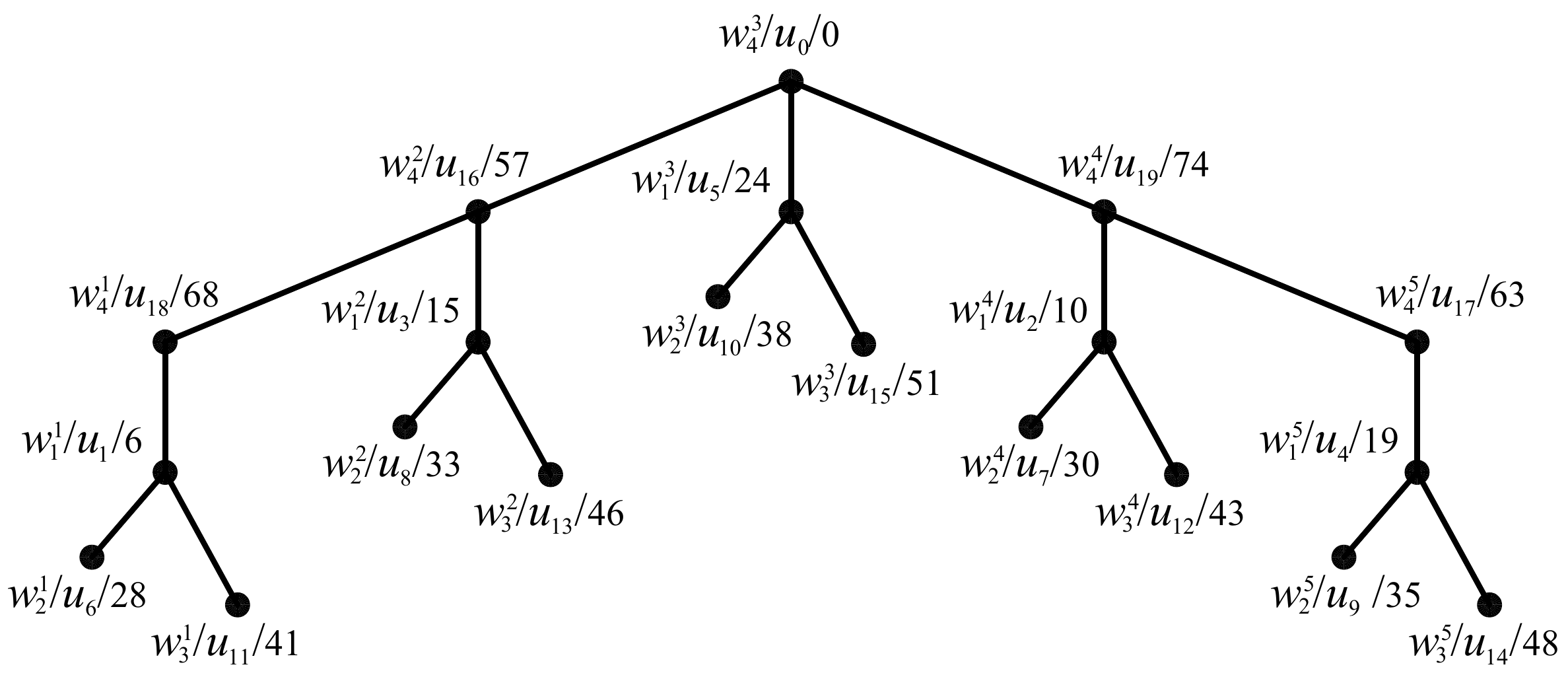}
\caption{\small An optimal radio labeling of $F(5,4)$ together with the corresponding ordering of vertices.}
\label{fig2A}
\end{figure}

\begin{figure}[ht]
\centering
\includegraphics*[scale=0.3]{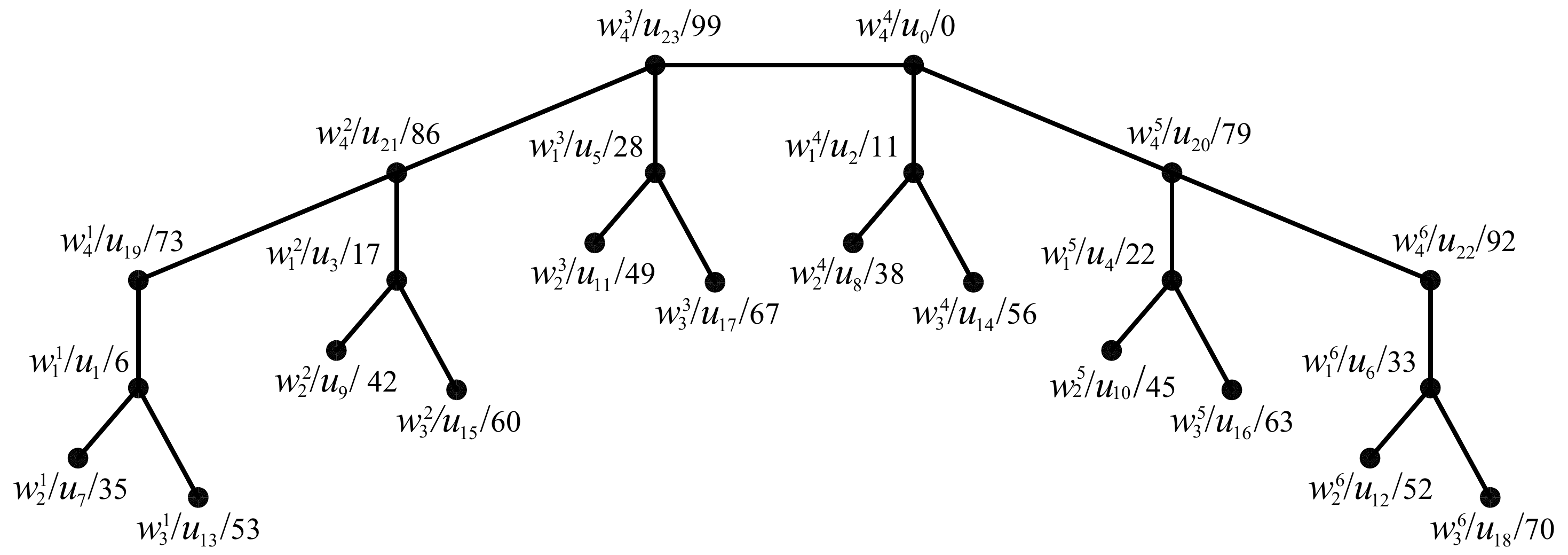}
\caption{\small An optimal radio labeling of $F(6,4)$ together with the corresponding ordering of vertices.}
\label{fig2B}
\end{figure}

\subsection{Caterpillars}

A tree is called a \emph{caterpillar} if the removal of all its degree-one vertices results in a path, called the \emph{spine}. Denote by $C(n, k)$ the caterpillar in which the spine has length $n-3$ and all vertices on the spine have degree $k$, where $n \ge 3$ and $k \ge 2$. Note that $\ve(C(n,k)) = 1$ when $n$ is odd and $\ve(C(n,k)) = 0$ when $n$ is even. Note also that $C(n, 2) = P_n$ is the path with $n$ vertices.

\begin{Theorem}
\label{thm:cater}
Let $n \geq 4$ and $k \ge 2$ be integers. Denote $\ve = \ve(C(n,k))$. Then
\begin{equation}
\label{rn:cater}
\rn(C(n,k)) = \frac{1}{2}\left((n-2)^{2}+\ve\right)(k-1)+n-1+\ve.
\end{equation}
\end{Theorem}

In the special case when $k = 2$, Theorem \ref{thm:cater} gives the following known result.

\begin{Corollary}
\label{coro:cater}
(\cite{Liu})
Let $n \geq 4$ be an integer. Then
\begin{equation}\label{rn:path}
\rn(P_{n}) =
\left\{
  \begin{array}{ll}
    2m^{2}+2, & \hbox{if }n=2m+1  \\[0.2cm]
    2m(m-1)+1, & \hbox{if }n=2m.
  \end{array}
\right.
\end{equation}
\end{Corollary}

The radio number of $C(n, k)$ for even $n$ was considered in \cite{KP}. However, the formula in \cite[Theorem 2.3]{KP} seems incorrect -- it is bigger by one than the actual value of $\rn(C(n, k))$ shown in \eqref{rn:cater}. (We have taken into account that the radio number defined in \cite{KP} is bigger than the usual definition by one.) For example, $\rn(C(10, 4) = 105$ as shown in Fig. \ref{Cater2}, while \cite[Theorem 2.3]{KP} gives $106$.

\medskip
\begin{proof}{\bf of Theorem \ref{thm:cater}}~
$C(n,k)$ has order $p = n+(n-2)(k-1)$, diameter $d = n-1$ and total level
$$
L(C(n,k)) =
\left\{
  \begin{array}{ll}
    \frac{(n^{2}-5)(k-1)}{4}+1, & \hbox{if }n=2m+1 \\[0.2cm]
    \frac{n(n-2)(k-1)}{4}, & \hbox{if }n=2m.
  \end{array}
\right.
$$
Plugging these into (\ref{eq:lb}), we obtain
\begin{equation}
\label{eq:clb}
\rn(C(n,k)) \geq \frac{1}{2}\left((n-2)^{2}+\ve\right)(k-1)+n-1.
\end{equation}
Denote by $v_{2} \ldots v_{n-1}$ the spine of $C(n,k)$. Choose $v_1$ and $v_n$ to be distinct degree-one vertices of $C(n,k)$ adjacent to $v_2$ and $v_{n-1}$, respectively. For each $2 \leq i \leq n-1$, denote by $v_{i,j}$, $1 \leq j \leq k-2$ the neighbours of $v_{i}$ not on the path $v_{1} v_{2} \ldots v_{n-1} v_n$.

\textsf{Case 1:} $n = 2m+1$ is odd.

In this case, $C(n,k)$ has only one weight centre, namely $v_{m+1}$, and so $\ve = 1$. We first prove \be
\label{eq:cub}
\rn(C(n,k)) \le \frac{1}{2}\left((n-2)^{2}+1\right)(k-1)+n
\ee
by using Theorem \ref{thm:ub1}. To this end we define a linear order $u_0, u_1, \ldots, u_{p-1}$ of the vertices of $C(n, k)$ as follows. Set $u_{0}$ = $v_{m}$, $u_{1}$ = $v_{2m}$, $u_{2}$ = $v_{1}$, $u_{3}$ = $v_{m+1}$, $u_{4}$ = $v_{2m+1}$ and $u_{p-1}$ = $v_{m+2}$. Relabel the remaining vertices on the spine by setting
\begin{equation}\label{ord1}
u_{t} = v_{i}, \mbox{ where } t =
\left\{
  \begin{array}{ll}
    2(m-i)+3, & \hbox{if } 2 \le i \le m-1  \\[0.2cm]
    2(2m-i+2), & \hbox{if } m+3 \le i < 2m.
  \end{array}
\right.
\end{equation}
We obtain $u_{5}, \ldots, u_{n-2}$ in this way. Set
\begin{equation*}
u_{t} = v_{i,j}, \mbox{ where } t =
\left\{
  \begin{array}{ll}
    2m+2(k-2)(i-3)+2j-1, & \hbox{if }3 \le i \le m,\ 1 \le j \le k-2  \\[0.2cm]
    2m+2(k-2)(i-m-2)+2(j-1), & \hbox{if }m+2 \le i \le 2m-1,\ 1 \le j \le k-2
  \end{array}
\right.
\end{equation*}
to obtain $u_{n-1}, \ldots, u_{p-3k+4}$. Finally, set
\begin{equation*}
u_{t} = v_{i,j}, \mbox{ where } t =
\left\{
  \begin{array}{ll}
    2m+2(k-2)(m-2)+3j-1, & \hbox{if }i = 2,\ 1 \le j \le k-2 \\[0.2cm]
    2m+2(k-2)(m-2)+3j-2, & \hbox{if }i = m+1,\ 1 \le j \le k-2 \\[0.2cm]
    2m+2(k-2)(m-2)+3(j-1), & \hbox{if }i = 2m,\ 1 \le j \le k-2
  \end{array}
\right.
\end{equation*}
to obtain $u_{p-3k+5}, \ldots, u_{p-1}$.

\textsf{Claim 1:} The linear order $u_{0}, u_{1}, \ldots, u_{p-1}$ above satisfies condition (\ref{eq:dij}).

In fact, denoting the right-hand side of (\ref{eq:dij}) with respect to the order above by $S_{i,j}$, we have $S_{i,j}$ = $\sum_{t=i}^{j-1}(L(u_{t})+L(u_{t+1})-2m-1)+2m+1$. It is easy to verify that (\ref{eq:dij}) is satisfied if $j = i+1$, or $i = 0$, or $i = p-1$. If $j \geq i+3$, then for $1 \le i < j \le n-2$, $u_{t}$ with $i \le t \le j-2$ satisfies (b) and (c) in Theorem \ref{thm:cor1} and hence (\ref{eq:dij}) is satisfied; for $i \le n-2 < j$ or $n-2 \le i < j \le p-3k+4$, we have $L(u_{t})+L(u_{t+1}) \leq m+2$ for $i \leq t \leq j-1$ and hence $S_{i,j} \leq (j-i)(m+2-(2m+1))+(2m+1) \leq 3(-m+1)+(2m+1) = -m+4 \leq 3 \leq d(u_{i},u_{j})$; for $i \le p-3k+4 < j \le p-2$ or $p-3k+4 < i < j \le p-2$, we have $S_{i,j} \le 1 \le d(u_{i},u_{j})$. Assume $j$ = $i+2 \geq 3$ in the remaining proof, so that $S_{i,j}$ = $L(u_{i})+2L(u_{i+1})+L(u_{i+2})-(2m+1)$. If $1 \le i \le n-4$, then $u_{t}$ with $i \le t \le j-2$ satisfies (b) and (c) in Theorem \ref{thm:cor1} and hence (\ref{eq:dij}) is satisfied. If $n - 4 < i \le n-2$, or $n-2 < i \le p-3k+2$, or $p-3k+3 \le i \le p-3k+4$ and $k \ge 3$, then $S_{i,j} \le 2 \leq d(u_{i},u_{j})$. If $p-3k+4 < i \le p-4$, then $L(u_{i}) + 2L(u_{i+1}) + L(u_{i+2}) \leq 3m+1$ and hence $S_{i,j} \leq m \leq d(u_{i},u_{j})$. This proves Claim 1.

Since $L(u_{0})+L(u_{p-1}) = 2$, we obtain \eqref{eq:cub} immediately from Theorem \ref{thm:ub1} and Claim 1.

In view of \eqref{eq:clb} and \eqref{eq:cub}, it remains to prove $\rn(C(n,k)) \ne \frac{1}{2}\left((n-2)^{2}+1\right)(k-1)+n-1$. Suppose otherwise. Then by Theorem \ref{thm:ub} there exists a linear order $u_{0}, u_{1}, \ldots, u_{p-1}$ of the vertices of $C(n,k)$ satisfying (\ref{eq:dij}) such that $L(u_{0})+L(u_{p-1}) = 1$ and $u_{i}$ and $u_{i+1}$ are in different branches. Denote by $T, T'$ the branches of $C(n,k)$ containing $v_1, v_n$, respectively. (Each of the other $k-2$ branches contains only one vertex.) Denote $S = \{u: u \in V(T), L(u) = m\}$ and $S^{'} = \{u: u \in V(T'), L(u) = m\}$. Then $|S| = |S^{'}| = k-1$. Since $L(u_{0})+L(u_{p-1}) = 1$ and span$(f)-f$ is an optimal radio labeling whenever $f$ is an optimal radio labeling, without loss of generality we may assume that $u_{0}$ = $v_{m+1}$ and $u_{p-1} \in N(u_{0})$.

Since $u_{0} = v_{m+1}$ and $u_{i}$ and $u_{i+1}$ are in different branches, there exists a vertex $u_{t} \in S$ such that $d(u_{t-1}, u_{t}) \geq m+a$, $d(u_{t}, u_{t+1}) \geq m + b$, $d(u_{t-1}) \neq 1$ and $d(u_{t+1}) \neq 1$ for some $a, b \geq 1$ with $a \neq b$. Hence $S_{t-1, t+1} = L(u_{t-1})+2L(u_{t})+L(u_{t+1})-(2m+1) = (m+a)+(m+b)-(2m+1) = a+b-1 > |a-b| = d(u_{t-1},u_{t+1})$, contradicting the assumption that (\ref{eq:dij}) is satisfied for any $0 \le i < j \le p-1$. Therefore, $\rn(C(n,k)) = \frac{1}{2}\left((n-2)^{2}+1\right)(k-1)+n$.

\textsf{Case 2:} $n = 2m$ is even.

In this case, $C(n,k)$ has two adjacent weight centres, namely $v_{m}$ and $v_{m+1}$, and so $\ve = 0$. It suffices to prove the existence of a linear order $u_{0}, u_{1}, \ldots, u_{p-1}$ of the vertices of $C(n,k)$ such that the conditions of Theorem \ref{thm:ub} are satisfied. Set $u_{0} = v_m$ and $u_{p-1} = v_{m+1}$. Set
\begin{equation}\label{ord2}
u_{t} = v_{i}, \mbox{ where } t =
\left\{
  \begin{array}{ll}
    2(m-i), & \hbox{if } 1 \le i \le m-1  \\[0.2cm]
    2(2m-i)+1, & \hbox{if } m+2 \le i \le 2m.
  \end{array}
\right.
\end{equation}
We obtain $u_1, \ldots, u_{n-2}$ in this way. Let
\begin{equation*}
u_{t} = v_{i,j}, \mbox{ where } t =
\left\{
  \begin{array}{ll}
    2m+2(k-2)(i-2)+2(j-1), & \hbox{if }i \leq m,\ 1 \le j \le k-2  \\[0.2cm]
    2m+2(k-2)(i-m-1)+2j-3, & \hbox{if }i>m,\ 1 \le j \le k-2.
  \end{array}
\right.
\end{equation*}
to obtain $u_{n-1}, \ldots, u_{p-2}$. Note that $\{u_{0}, u_{p-1}\} = \{v_m, v_{m+1}\} = W(C(n,k))$ and $u_{i}$ and $u_{i+1}$ are in opposite branches for $1 \leq i \leq p-1$. It remains to prove the following:

\textsf{Claim 2:} The linear order $u_{0}, u_{1}, \ldots, u_{p-1}$ above satisfies condition (\ref{eq:dij}).

In fact, denoting the right-hand side of (\ref{eq:dij}) with respect to the order above by $S_{i,j}$, we have $S_{i,j} = \sum_{t=i}^{j-1}(L(u_{t})+L(u_{t+1})-2m+1)+2m$. It is easy to verify that (\ref{eq:dij}) is satisfied when $j$ = $i+1$, or $i$ = 0, or $i$ = $p-1$. If $j \geq i+3$, then for $n = 4$ we have $L(u_{t})+L(u_{t+1}) \leq m$ for $i \leq t \leq j-1$ and hence $S_{i,j} \leq (j-i)(m-(2m-1))+2m \leq 3(-m+1)+2m = -m+3 \leq 1 \leq d(u_{i},u_{j})$; and for $n \geq 6$ we have $L(u_{t})+L(u_{t+1}) \leq m+1$ for $i \leq t \leq j-1$ and hence $S_{i,j} \leq (j-i)(m+1-(2m-1))+2m \leq 3(-m+2)+2m = -m+6 \leq 3 \leq d(u_{i},u_{j})$. If $j$ = $i+2 \geq 3$, then $S_{i,j} = L(u_{i})+2L(u_{i+1})+L(u_{i+2})-2m+2$. If both $u_{i}$ and $u_{j}$ are on the path $v_{1} v_{2} \ldots v_{n}$, then they satisfy (b) and (d) in Theorem \ref{thm:cor1} and hence satisfy (\ref{eq:dij}). If $u_{i}$ is on the path $v_{1} v_{2} \ldots  v_{n}$ but $u_{j}$ is not, then $S_{i,j} = 2 \leq d(u_{i},u_{j})$ since $L(u_{t})+L(u_{t+1}) \leq m$ for every $t$. If neither $u_{i}$ nor $u_{j}$ is on the path $v_{1} v_{2} \ldots v_{n}$, then either $L(u_{i})+2L(u_{i+1})+L(u_{i+2}) \leq 2m$ or $L(u_{i})+2L(u_{i+1})+L(u_{i+2}) \leq 2m+1$, and hence $S_{i,j} = 2 \leq d(u_{i},u_{j})$ or $S_{i,j} = 3 \leq d(u_{i},u_{j})$, respectively. This completes the proof of Claim 2.

So far we have completed the proof of (\ref{rn:cater}). Moreover, by Theorems \ref{thm:ub} and \ref{thm:ub1}, the labeling given by (\ref{eq:f0})-(\ref{eq:f}) with respect to the linear order above is an optimal radio labeling of $C(n,k)$.
\qed
\end{proof}

The reader is referred to Fig. \ref{Cater1} and \ref{Cater2} for an illustration of naming, ordering and labeling of the vertices of $C(9,4)$ and $C(10,4)$ by using the procedure in the proof of Theorem \ref{thm:cater}.
\begin{figure}[ht]
\centering
\includegraphics*[scale=0.3]{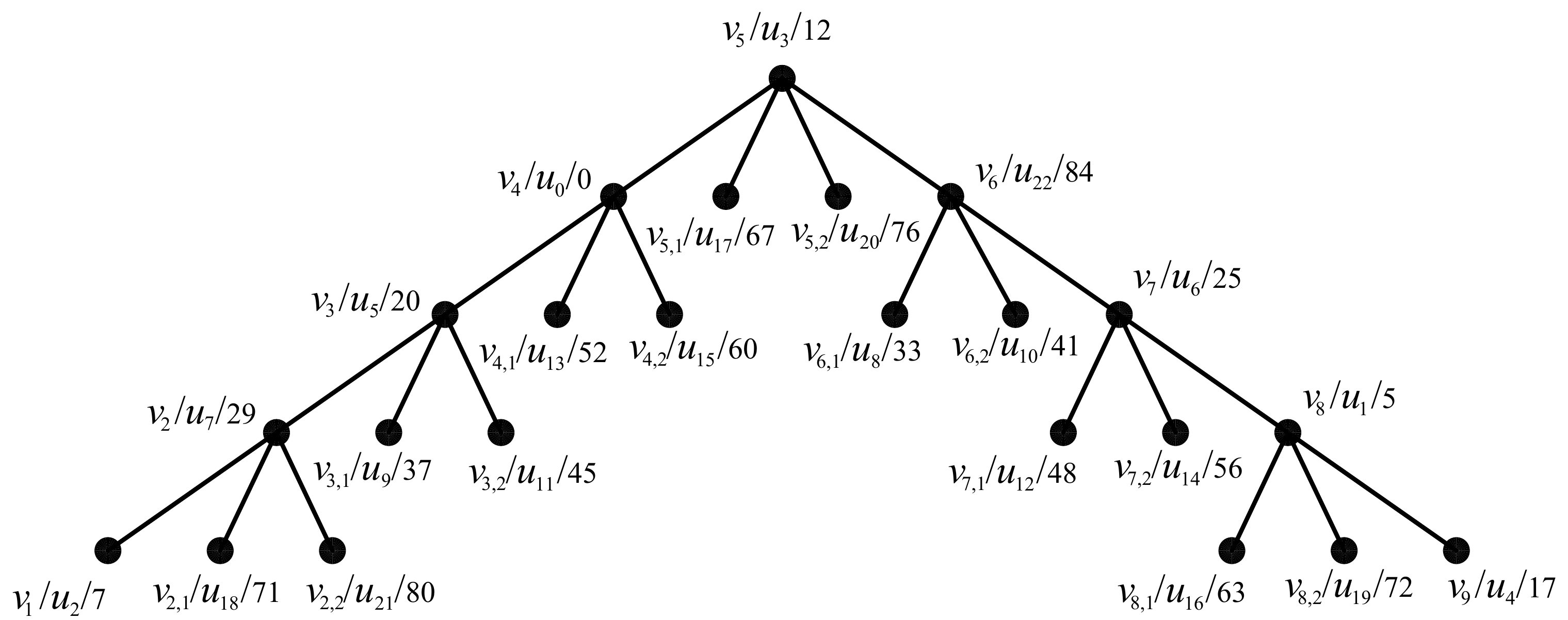}
\caption{\small An optimal radio labeling of $C(9,4)$ together with the corresponding ordering of vertices.}
\label{Cater1}
\end{figure}

\begin{figure}[ht]
\centering
\includegraphics*[scale=0.3]{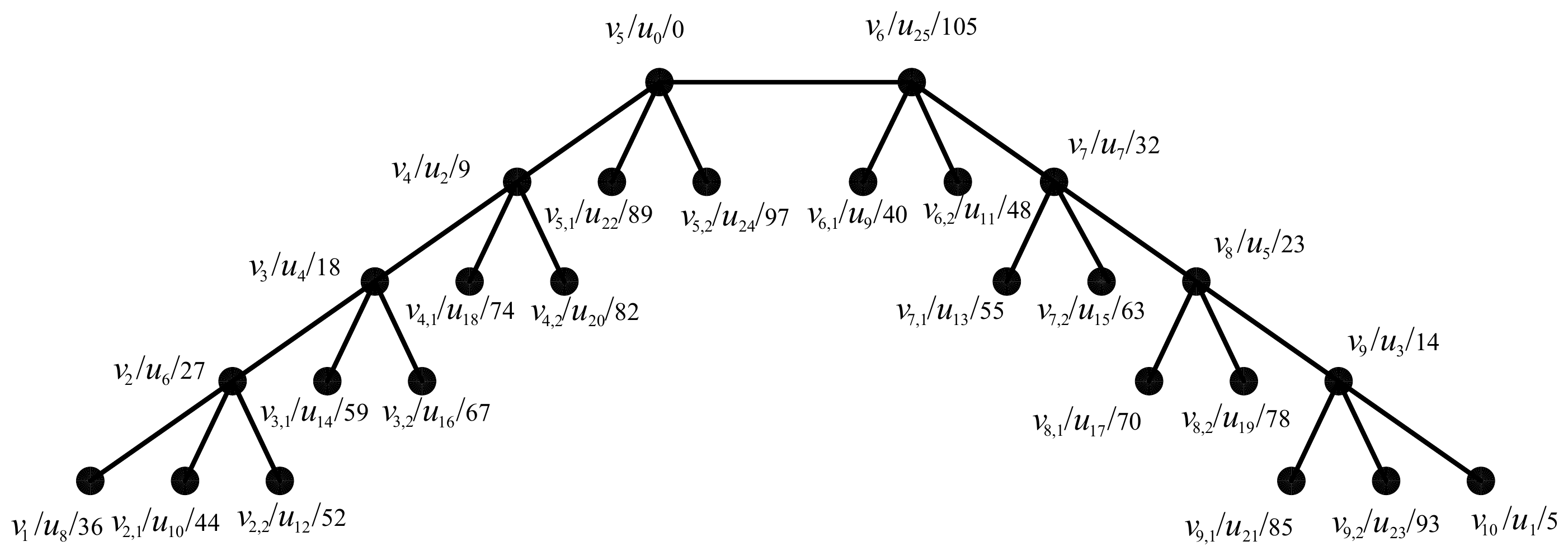}
\caption{\small An optimal radio labeling of $C(10,4)$ together with the corresponding ordering of vertices.}
\label{Cater2}
\end{figure}

\section{Concluding remarks}
\label{sec:rem}

It is well known that the centre of any tree $T$ consists of one vertex $r$ or two adjacent vertices $r, r'$, depending on whether $\diam(T)$ is even or odd. We may think of $T$ as a rooted tree with root $r$ or $\{r, r'\}$, respectively. Hal\'{a}sz and Tuza \cite{Tuza} defined a \emph{level-wise regular tree} to be a tree $T$ in which all vertices at distance $i$ from root $r$ or $\{r, r'\}$ have the same degree, say $m_{i}$, for $0 \le i \le h$, where $h$ is the \emph{height} of $T$, namely the largest distance from a vertex to the root. If $m_{0} = m_{1} = \cdots = m_{h-1} = m$ and $m_{h} = 1$, then $T$ is called an \emph{internally $m$-regular complete tree} \cite{Tuza}. It can be verified that the centre and the weight centre of such a tree are identical.

Hal\'{a}sz and Tuza \cite[Theorem 1]{Tuza} proved that the radio number of the internally $(m+1)$-regular complete tree $T$ with diameter $d$ and height $h = \lfloor d/2 \rfloor$, where $d, m \geq 3$, is given by
\begin{equation}\label{rn:levelT}
\rn(T) =
\left\{
  \begin{array}{ll}
    m^{h}+\frac{4m^{h+1}-2hm^{2}-4m+2h}{(m-1)^{2}}, & \hbox{if }d=2h  \\[0.2cm]
    2m^{h}+\frac{6m^{h+1}-2m^{h}-(2h+1)m^{2}-4m+2h+1}{(m-1)^{2}}, & \hbox{if }d=2h+1.
  \end{array}
\right.
\end{equation}
This result can be proved by using Theorem \ref{thm:ub}, as shown independently in an earlier version  of the present paper. (An extended abstract of that version can be found in \cite{BVZ}.) A sketch of our proof is as follows.

The order of $T$ is
$$
p = \left\{
\begin{array}{ll}
1+\frac{m+1}{m-1} (m^{h}-1), & \mbox{if } d = 2h \\ [0.2cm]
2 \left(1 + \frac{m}{m-1} (m^{h}-1)\right), & \mbox{if } d = 2h+1.
\end{array}
\right.
$$
Using $1 + 2x + 3x^{2} + \cdots + nx^{n-1} = \frac{nx^{n}}{x-1} - \frac{x^{n}-1}{(x-1)^{2}}$, one can see that the total level of $T$ is
$$
L(T) =\left\{
\begin{array}{ll}
(m+1)\left(\frac{h m^{h}}{(m-1)} - \frac{m^{h}-1}{(m-1)^{2}}\right), & \mbox{if } d = 2h \\ [0.2cm]
2m\left(\frac{h m^{h}}{(m-1)} - \frac{m^{h}-1}{(m-1)^{2}}\right), & \mbox{if } d = 2h+1.
\end{array}
\right.
$$
Using these, one can verify that for $T$ the right-hand sides of (\ref{eq:ub}) and (\ref{rn:levelT}) are equal. Thus to prove (\ref{rn:levelT}) it suffices to prove the existence of an linear order of the vertices of $T$ satisfying the conditions of Theorem \ref{thm:ub}.

\textsf{Case 1:} $T$ has only one central vertex, say $w$.

In this case, $w$ is the unique weight centre of $T$. Denote the children of $w$ by $w^{1}, w^{2}, \ldots, w^{m+1}$. Denote the $m$ children of each $w^{t}$ by $w_{0}^{t}, w_{1}^{t}, \ldots, w_{m-1}^{t}$, $1 \le t \le m+1$. Denote the $m$ children of each $w_{i}^{t}$ by $w_{i0}^{t}, w_{i1}^{t}, \ldots, w_{i(m-1)}^{t}$, $0 \leq i \leq m-1$, $1 \leq t \leq m+1$. Inductively, denote the $m$ children of $w_{i_{1},i_{2},\ldots,i_{l}}^{t}$ ($0 \leq i_{1}, i_{2}, \ldots, i_{l} \leq m-1$, $1 \leq t \leq m+1$) by $w_{i_{1}, i_{2},\ldots,i_{l}, i_{l+1}}^{t}$ where $0 \le i_{l+1} \le m-1$. Continue this until all vertices of $T$ are indexed this way. Rename the vertices of $T$ as follows: For $1 \leq t \leq m+1$, set
$$
v_{j}^{t} := w_{i_{1},i_{2}, \ldots, i_{l}}^{t},\;\, \mbox{where}\;\, j = 1 + i_{1} + i_{2}m + \cdots + i_{l}m^{l-1} + \sum_{l+1 \leq t \leq \lfloor d/2 \rfloor} m^{t}.
$$

Set $u_{0} := w$. For $1 \leq j \leq p-m-2$, let
$$
u_{j} := \left\{
\begin{array}{ll}
v_{s}^{t},\;\,\mbox{where $s = \lceil j/(m+1) \rceil$}, & \mbox{if $j \equiv t$ (mod $(m+1)$) for some $t$ with $1 \le t \le m$} \\ [0.3cm]
v_{s}^{m+1},\;\,\mbox{where $s = \lceil j/(m+1) \rceil$}, & \mbox{if $j \equiv 0$ (mod $(m+1)$)}.
\end{array}
\right.
$$
Let
$$
u_{j} := w^{j-p+m+2},\;\, p-m-1 \leq j \leq p-1.
$$
Note that $u_{p-1} = w^{m+1}$ is adjacent to $w$. Note also that $u_{i}$ and $u_{i+1}$ are in different branches so that $\phi(u_{i},u_{i+1}) = 0$, for $1 \leq i \leq p-2$.

\textsf{Case 2:} $T$ has two adjacent central vertices, say $w$ and $w'$.

In this case, $w$ and $w'$ are also weight centres of $T$. Denote the neighbours of $w$ other than $w'$ by $w_{0}, w_{1}, \ldots, w_{m-1}$ and the neighbours of $w'$ otherwise than $w$ by $w'_{0}, w'_{1}, \ldots, w'_{m-1}$. For $0 \le i \le m-1$, denote the $m$ children of each $w_{i}$ (respectively, $w'_{i}$) by $w_{i0}, w_{i1}, \ldots, w_{i(m-1)}$ (respectively, $w'_{i0}, w'_{i1}, \ldots, w'_{i(m-1)}$). Inductively, for $0 \leq i_{1},i_{2},\ldots,i_{l} \leq m-1$, denote the $m$ children of $w_{i_{1},i_{2}, \ldots, i_{l}}$ (respectively, $w'_{i_{1},i_{2},\ldots,i_{l}}$) by $w_{i_{1},i_{2},\ldots,i_{l},i_{l+1}}$ (respectively, $w'_{i_{1},i_{2},\ldots,i_{l},i_{l+1}}$), where $0 \leq i_{l+1} \leq m-1$. Rename
$$
v_{j} := w_{i_{1},i_{2},\ldots,i_{l}},\;\, v'_{j} := w'_{i_{1},i_{2},\ldots,i_{l}},\;\, \mbox{where}\;\, j = 1 + i_{1} + i_{2}m + \cdots + i_{l}m^{l-1} + \sum_{l+1 \leq t \leq \lfloor d/2 \rfloor} m^{t}.
$$

Let $u_0 := w$ and $u_{p-1} := w'$. For $1 \leq j \leq p-2$, let
\bean
u_j := \left\{
\begin{array}{ll}
v_{s},\;\,\mbox{where $s = \lceil j/2 \rceil$}, & \mbox{if $j \equiv 0$ (mod $2$)}\\ [0.3cm]
v'_{s},\;\,\mbox{where $s = \lceil j/2 \rceil$}, & \mbox{if $j \equiv 1$ (mod $2$)}.
\end{array}
\right.
\eean
Then $u_{i}$ and $u_{i+1}$ are in opposite branches for $1 \leq i \leq p-2$, and $u_{i+2j}$, $j = 0,1, \ldots, (m-1)$  are in different branches for $1 \leq i \leq p-2m+1$, so that $\phi(u_{i},u_{i+1}) = 0$ and $\delta(u_{i},u_{i+1}) = 1$.

In each case above, it can be verified that the linear order $u_{0}, u_{1}, \ldots, u_{p-1}$ satisfies the conditions of Theorem \ref{thm:ub} (details are omitted). So (\ref{rn:levelT}) follows.

The reader is referred to Fig. \ref{Fig3} and \ref{Fig4New} for an illustration of naming, ordering and labeling of the vertices of two internally $3$-regular complete trees of hight 3 by using the above procedure.

\begin{figure}[ht]
\centering
\includegraphics*[scale=0.35]{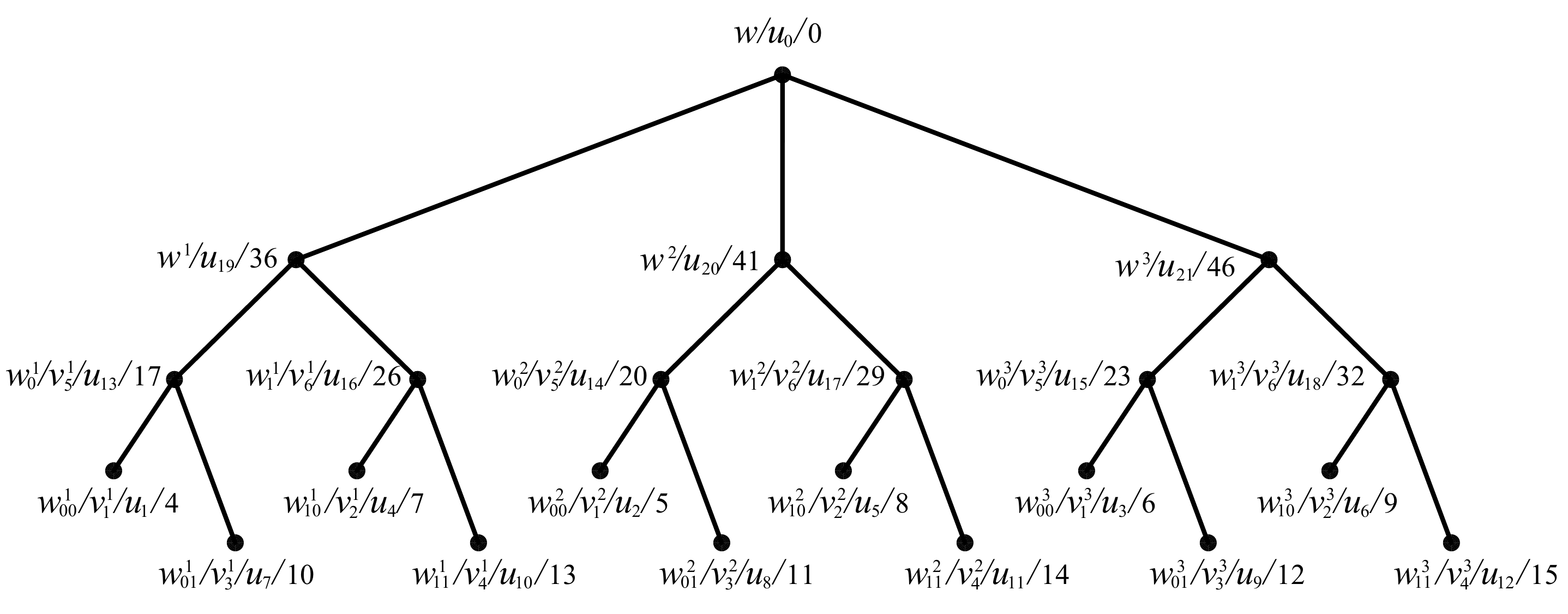}
\caption{\small An optimal radio labeling of the internally 3-regular complete tree with hight 3 and one central vertex.}
\label{Fig3}
\end{figure}

\begin{figure}[ht]
\centering
\includegraphics*[scale=0.35]{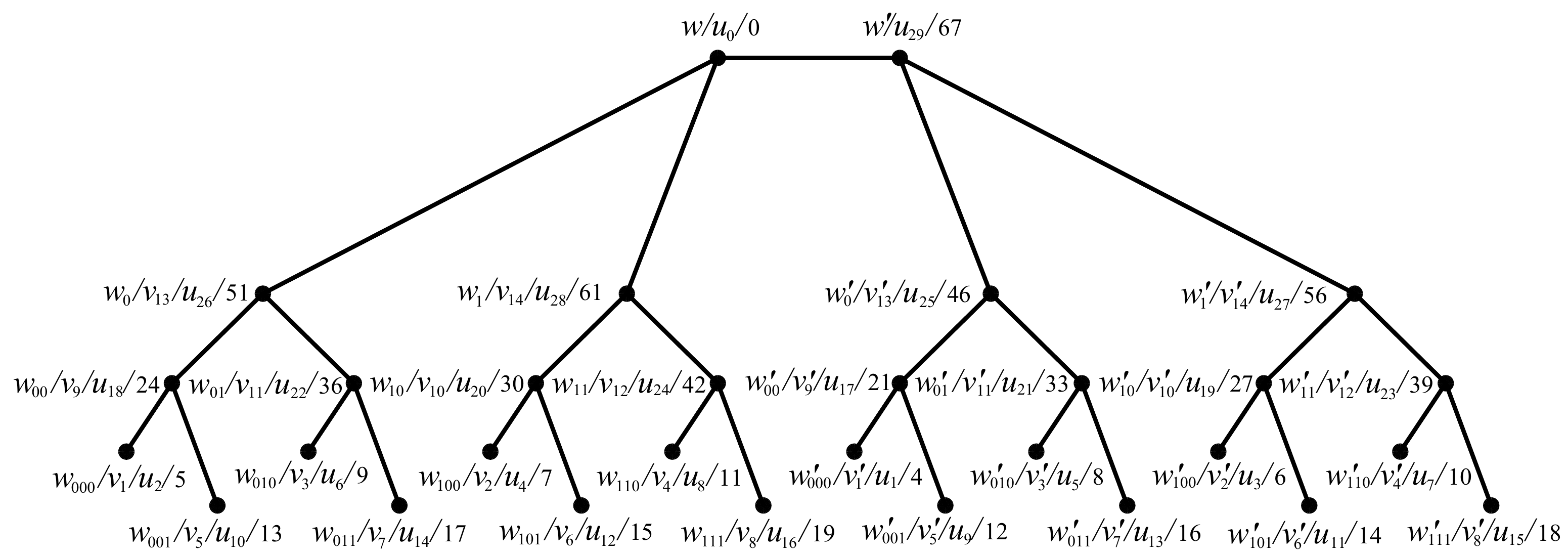}
\caption{\small An optimal radio labeling of the internally 3-regular complete tree with hight 3 and two central vertices.}
\label{Fig4New}
\end{figure}

A \emph{complete $m$-ary tree of hight $k$}, denoted by $T_{k,m}$, is a rooted tree such that each vertex other than leaves has $m$ children and all leaves are at distance $k$ from the root. Li \emph{et al.} \cite[Theorem 2]{Li} proved that, for $m \geq 3$ and $k \geq 2$,
\be
\label{rn:tkm}
\rn(T_{k,m}) = \frac{m^{k+2}+m^{k+1}-2km^{2}+(2k-3)m+1}{(m-1)^{2}}.
\ee
This result can also be proved by using Theorem \ref{thm:ub}. The order, diameter and total level of $T_{k,m}$ are $p = (m^{k+1}-1)/(m-1)$, $d = 2k$ and $L(T_{k,m}) = (km^{k+2}-(k+1)m^{k+1}+m)/(m-1)^{2}$, respectively. Plugging these into (\ref{eq:ub}), one can verify that the right-hand sides of (\ref{eq:ub}) and \eqref{rn:tkm} are identical. It can be verified that the linear order $u_{0}, u_{1}, \ldots, u_{p-1}$ of the vertices of $T_{k,m}$ given in \cite[Section 4.1]{Li} satisfies the conditions of Theorem \ref{thm:ub}. Thus we obtain \eqref{rn:tkm} by Theorem \ref{thm:ub}.

\section*{Acknowledgements}
We appreciate the two anonymous referees for their helpful comments and careful reading. The third author was supported by the Australian Research Council (FT110100629).

\end{document}